\documentclass{daj}

\usepackage[shortlabels]{enumitem}
\setenumerate[1]{label={(\arabic*)}}

\usepackage{tikz,amsmath,amsfonts,amssymb,amsthm}
\newtheorem{theorem}{Theorem}[section]		
\newtheorem{lemma}[theorem]{Lemma}
\newtheorem{claim}[theorem]{Claim}
\newtheorem{proposition}[theorem]{Proposition}
\newtheorem{problem}[theorem]{Problem}
\newcommand{\beq}[1]{\begin{equation}\label{#1}}
\newcommand{\enq}[0]{\end{equation}}

\def\Prob{\mathbb{P}}
\def\E{\mathbb{E}}
\def\N{\mathbb{N}}
\def\LL{\mathcal{L}}
\def\SS{\mathcal{S}}
\def\mathscr{\mathcal}
\def\HH{\mathcal{H}}
\def\BB{\mathcal{B}}

\DeclareMathOperator\Deg{d}
\DeclareMathOperator\Tr{Tr}

\newcommand{\eps}{\ensuremath{\varepsilon}}

\let\originalleft\left
\let\originalright\right
\renewcommand{\left}{\mathopen{}\mathclose\bgroup\originalleft}
\renewcommand{\right}{\aftergroup\egroup\originalright}

\dajAUTHORdetails{%
  title = {Simplicial Homeomorphs and Trace-Bounded Hypergraphs}, 
  author = {Jason Long, Bhargav Narayanan, and Corrine Yap},
  plaintextauthor = {Jason Long, Bhargav Narayanan, Corrine Yap},
}   

\dajEDITORdetails{%
   year={2022},
   number={6},
   received={28 December 2020},   
   published={7 July 2022},  
   doi={10.19086/da.36647},       
}   

\begin{document}

\begin{frontmatter}[classification=text]

\title{Simplicial Homeomorphs and Trace-Bounded Hypergraphs} 

\author[jl]{Jason Long}
\author[bn]{Bhargav Narayanan\thanks{Supported in part by NSF grants DMS-1800521 and CCF-1814409}}
\author[cy]{Corrine Yap\thanks{Supported in part by NSF grant DMS-1800521}}

\begin{abstract}
Our first main result is the following basic fact about simplicial complexes: for each $k \in \N$, there exists an exponent $\lambda_k \ge k^{-2k^2}$ such that for any $k$-complex $\SS$, every $k$-complex on $n \ge n_0(\SS)$ vertices with at least $n^{k+1 - \lambda_k}$ facets contains a homeomorphic copy of $\SS$. The existence of these exponents was suggested by Linial in 2006 but was previously known only in dimensions one and two, both by highly dimension-specific arguments: the existence of $\lambda_1$ is a result of Mader from 1967, and the existence of $\lambda_2$ was established by Keevash--Long--Narayanan--Scott in 2020. We deduce this geometric theorem from a purely combinatorial result about trace-bounded hypergraphs, where an $r$-partite $r$-graph $H$ with partition classes $V_1, V_2, \dots, V_r$ is said to be $d$-trace-bounded if for each $2 \le i \le r$, all the vertices of $V_i$ have degree at most $d$ in the trace of $H$ on $V_1 \cup V_2 \cup \dots \cup V_i$. Our second main result is the following fact about degenerate trace-bounded hypergraphs: for all $r \ge 2$ and $d\in\N$, there exists an exponent $\alpha_{r,d} \ge (5rd)^{1-r}$ such that for any $d$-trace-bounded $r$-partite $r$-graph $H$, every $r$-graph on $n \ge n_0(H)$ vertices with at least $n^{r - \alpha_{r,d}}$ edges contains a copy of $H$. This strengthens a theorem of Conlon--Fox--Sudakov from 2009 who showed that a similar result holds for $r$-partite $r$-graphs $H$ satisfying the stronger hypothesis that the vertex-degrees in all but one of its partition classes are bounded (in $H$, as opposed to in its traces).
\end{abstract}
\end{frontmatter}

\section{Introduction}
This paper aims to answer the following basic geometric question about $k$-dimensional simplicial complexes (or \emph{$k$-complexes} for short) that arises in  the `high-dimensional combinatorics' programme of Linial~\cite{linial4, linial3}.
\begin{problem}\label{kgraphhomeo}
Given a $k$-complex $\SS$, how many facets can a $k$-complex on $n$ vertices have if it contains no homeomorphic copy of $\SS$?
\end{problem}
For a $k$-complex $\SS$, let $\lambda(\SS)$ be the supremum over all $\lambda$ for which the maximum number of facets in a $k$-complex on $n$ vertices with no homeomorphic copy of $\SS$ is $O(n^{k+1 - \lambda})$. It is essentially folklore --- see~\cite{myhomeo} for a discussion --- that $\lambda(\SS) > 0$ for every $k$-complex $\SS$. A much more intriguing possibility, namely that for every $k$-complex $\SS$, $\lambda(\SS)$ is bounded below uniformly by some universal exponent $\lambda_k > 0$ that depends only on the dimension $k$, was suggested by Linial~\cite{linialques1, linialques2} (explicitly for dimension two and implicitly for higher dimensions); our first main result establishes this in every dimension.

\begin{theorem}\label{mainthm1}
	For all $k \in \N$, there is a $\lambda_k \ge k^{-2k^2}$ such that for any $k$-complex $\SS$, every $k$-complex on $n \ge n_0(\SS)$ vertices with at least $n^{k+1 - \lambda_k}$ facets contains a homeomorphic copy of $\SS$.
\end{theorem}

The existence of such universal exponents $\lambda_k$ as in Theorem~\ref{mainthm1} was previously known only for $k = 1$ and $k=2$; that the optimal value of $\lambda_1$ is $1$ is a classical result of Mader~\cite{mader1}, and that $\lambda_2 \ge 1/5$ was shown recently by Keevash, Scott and the first and second authors~\cite{myhomeo}. The conjecturally optimal value of $\lambda_2$ is $1/2$, and establishing this remains open, though a beautiful recent result of Kupavskii, Polyanskii, Tomon, and Zakharov~\cite{surfaces} establishes that $\lambda(\SS) = 1/2$ whenever $\SS$ is the triangulation of any closed orientable two-dimensional surface, generalising a classical result of Brown, Erd\H{o}s and S\'os~\cite{bes} establishing this fact for the two-sphere.

It is worth mentioning that all of~\cite{bes, myhomeo, surfaces, mader1} proceed via arguments that are highly specific to dimensions one and two. Indeed, the main novelty in the proof of Theorem~\ref{mainthm1} is our ability to simultaneously handle all dimensions; this generality comes at a cost, however, since the resulting bounds in low dimensions are not very competitive with those in the aforementioned results.

We shall deduce Theorem~\ref{mainthm1} from a purely combinatorial result, of some independent interest, about the Tur\'an numbers of a large class of $r$-uniform hypergraphs (or \emph{$r$-graphs}, for short). For an $r$-partite $r$-graph $H$, let $\alpha(H)$ be the supremum over all $\alpha$ for which the maximum number of edges in an $r$-graph on $n$ vertices with no copy of $H$ is $O(n^{r - \alpha})$. A well-known result of Erd\H{o}s~\cite{degen} says that $\alpha(H) > 0$ for every $r$-partite $r$-graph $H$; this value $\alpha(H)$ is called the \emph{Tur\'an exponent} of $H$, and the determination of these exponents is the central problem --- see~\cite{turan3, turan2} --- of degenerate Tur\'an theory.

To state our second result, we need a definition. We say that an $r$-partite $r$-graph $H$ with partition classes $V_1, V_2, \dots, V_r$ is \emph{$d$-trace-bounded} if for each $2 \le i \le r$, all the vertices of $V_i$ have degree at most $d$ in the trace of $H$ on $V_1 \cup V_2 \cup \dots \cup V_i$. Our second main result establishes the existence of universal lower bounds on the Tur\'an exponents of degenerate trace-bounded hypergraphs.

\begin{theorem}\label{mainthm2}
	For all $r \ge 2$ and $d\in\N$, there is an $\alpha_{r,d} \ge (5rd)^{1-r}$ such that for any $d$-trace-bounded $r$-partite $r$-graph $H$, every $r$-graph on $n \ge n_0(H)$ vertices with at least $n^{r - \alpha_{r,d}}$ edges contains a copy of $H$.
\end{theorem}

Theorem~\ref{mainthm2} generalises a result of Conlon, Fox and Sudakov~\cite{CFS} which asserts, for all $r\ge 2$ and $d \in \N$, the existence of exponents $\alpha'_{r,d} > 0$ (of order roughly $(rd)^{1-r}$ as well) with the following property: for any $r$-partite $r$-graph $H$ with partition classes $V_1, V_2, \dots, V_r$ in which the degrees of the vertices in each of $V_2, V_3, \dots, V_r$ are at most $d$ in $H$, we have $\alpha(H) \ge \alpha'_{r,d}$. It is easy to see that any $r$-partite $r$-graph $H$ to which the aforementioned result applies is $d$-trace-bounded as well, so Theorem~\ref{mainthm2} clearly implies this result. Of course, Theorem~\ref{mainthm2} is genuinely stronger than the result in~\cite{CFS} since not every trace-bounded $r$-partite $r$-graph has bounded vertex-degrees in all but one of its partition classes, and indeed, the full strength of Theorem~\ref{mainthm2} will be crucial in proving Theorem~\ref{mainthm1}.

Two more remarks about Theorem~\ref{mainthm2} are in order. First, the fact that the optimal value of $\alpha_{2,d}$ is $1/d$ (as opposed to the $1/(10d)$ promised by Theorem~\ref{mainthm2}) is a result of Alon, Krivelevich and Sudakov~\cite{AKS} which may also be read out of earlier work of F\"uredi~\cite{Furedi}. Second, it is known that, in a sense, something like the trace-boundedness hypothesis in Theorem~\ref{mainthm2} is necessary if one expects to control the Tur\'an exponent in terms of vertex-degrees; indeed, from~\cite{KMV}, we know that for every $\eps > 0$, there exists a 3-partite 3-graph $H$ with all the vertex-degrees in one of its partition classes being 1 for which $\alpha(H) \le \eps$.

Let us summarise the discussion above by specialising to $3$-graphs. For a 3-partite 3-graph $H$ with partition classes $V_1$, $V_2$ and $V_3$, we have the following conclusions about its Tur\'an exponent $\alpha(H)$, listed below in order of decreasing strength of the hypotheses on $H$.
\begin{enumerate}
	\item If the degrees of the vertices in both $V_3$ and $V_2$ are bounded above by $d$ in $H$, then~\cite{CFS} shows that $\alpha(H)\ge  1/(15d)^2$.
	\item If the degrees of the vertices in $V_3$ are bounded above by $d$ in $H$, and the degrees of the vertices in $V_2$ are bounded above by $d$ in the trace of $H$ on $V_1 \cup V_2$, then Theorem~\ref{mainthm2} says that $\alpha(H) \ge 1/(15d)^2$.
	\item If all we know is that the degrees of the vertices in $V_3$ are bounded above by $d$ in $H$, then~\cite{KMV} shows that $\alpha(H)$ need not be bounded below uniformly in terms of $d$, even when $d = 1$.
\end{enumerate}

This paper is organised as follows. We begin by establishing some notation and making precise some of the undefined terminology appearing in the introduction in Section~\ref{sec:prelim}. The deduction of Theorem~\ref{mainthm1} from Theorem~\ref{mainthm2} is given in Section~\ref{sec:homs}, and the proof of Theorem~\ref{mainthm2} follows in Section~\ref{sec:proof}. We conclude with a discussion of some open problems in Section~\ref{sec:conc}.


\section{Preliminaries}\label{sec:prelim}
We shall only consider \emph{homogeneous} $k$-complexes, namely $k$-complexes all of whose facets are $k$-dimensional. Consequently, we may specify a $k$-complex $\SS$ over a vertex set $V$ by listing the family $F$ of its $k$-dimensional facets, each of which is a subset of $V$ of cardinality $k+1$ (though $\SS$ is, strictly speaking, the family of all subsets of its facets). We say that a $k$-complex $\mathscr{T}$ contains a \emph{homeomorph} (or a \emph{homeomorphic copy}) of a $k$-complex $\SS$ if there is a subcomplex of $\mathscr{T}$ that is homeomorphic to $\SS$. An $r$-graph $G$ on a vertex set $V$ is a family $E$ of $r$-element subsets of $V$ called the edges of $G$. A $k$-complex $\SS$ may hence be identified with a $(k+1)$-graph $G$ by viewing the facets of $\SS$ as the edges of $G$, and vice versa. When we specify a $k$-complex by its set of facets alone, or an $r$-graph by its edge set alone, the vertex set of the $k$-complex or the $r$-graph in question is taken to be the span of the facets or the edges respectively.

Since most of the work here will be in proving Theorem~\ref{mainthm2}, let us set out some more notation for working with an $r$-graph $G$. We write $v(G)$ and $e(G)$ for the number of vertices and edges of $G$ respectively. The \emph{link $\LL(v, G)$} of a vertex $v\in V(G)$ in $G$ is the $(r-1)$-graph whose edges are those sets $S$ for which $\{v\} \cup S$ is an edge of $G$, and the \emph{degree $\Deg(v, G)$} of $v$ is the number of edges of $G$ containing $v$, or equivalently $\Deg(v, G) = e(\LL(v, G))$. For an $(r-1)$-graph $J$ with $V(J) \subset V(G)$, its \emph{common neighbourhood $\Gamma(J, G)$} in $G$ is the set of vertices $v \in V(G)$ for which $\{v\} \cup S$ is an edge of $G$ for each edge $S \in E(J)$. Finally, for a subset $U \subset V(G)$ of the vertex set of $G$, the \emph{trace $\Tr(G, U)$} of $G$ on $U$ is the family $\{S \cap U: S \in E(G)\}$.

An $r$-graph $G$ is \emph{$r$-partite} if its vertex set admits a partition $V(G) = V_1 \cup V_2 \cup \dots \cup V_r$ such that every edge of $G$ contains exactly one vertex each from each of these $r$ partition classes. When $G$ is an $r$-partite $r$-graph with partition classes $V_1, V_2, \dots, V_r$, we abbreviate $\Tr(G, V_1\cup V_2 \cup \dots \cup V_i)$ by $\Tr_i(G)$, noting that $\Tr_i(G)$ is an $i$-graph for each $1 \le i \le r$. Finally, we say that an $r$-partite $r$-graph $G$ with partition classes $V_1, V_2, \dots, V_r$ is \emph{$d$-trace-bounded} if for each $2 \le i \le r$, we have
$\Deg(v, \Tr_i(G)) \le d$ for each $v \in V_i$.

It will be convenient for us to work with a large $r$-partite subgraph of a given $r$-graph; the following fact facilitates this, and follows from an easy averaging argument.

\begin{proposition}\label{triedges} Any $r$-graph on $rn$ vertices with $m$ edges contains an $r$-partite subgraph with partition classes each of size $n$ and at least $(r! / r^r)m$ edges. \qed
\end{proposition}

\section{Homeomorphs}\label{sec:homs}
Our proof of Theorem~\ref{mainthm1} relies on the following construction. Given a $k$-complex $\SS$, the \emph{canonical subdivision of $\SS$} is a $k$-complex $\tilde \SS$ homeomorphic to $\SS$ constructed as follows. The vertex set of $\tilde \SS$ is given by
\[ V(\tilde \SS) = V(\SS) \cup \{v_T: T \subset V(\SS), |T| \ge 2, \text{ and $T$ is contained in some facet of $\SS$} \};\]
in other words, we start with $V(\SS)$ and for each $2 \le t \le k+1$, we introduce a new vertex $v_T$ for each $t$-set $T$ contained in some facet of $\SS$. The facets $F(\tilde \SS)$ of $\tilde \SS$ are then obtained by subdividing each facet of $\SS$ into $(k+1)!$ facets as follows: for each facet $S \in F(\SS)$ of $\SS$, consider the $(k+1)!$ possible chains
\[ \{v\} \subsetneq T_2 \subsetneq T_3 \subsetneq \dots \subsetneq T_k \subsetneq S \]
with $v$ a vertex of $\SS$, and include $\{v, v_{T_2}, v_{T_3}, \dots, v_{T_k}, v_S\}$ in $F(\tilde \SS)$. It is not hard to see that $\tilde \SS$ is homeomorphic to $\SS$, as illustrated in Figure~\ref{fig:subd}.

\begin{proof}[Proof of Theorem~\ref{mainthm1}]
	We shall prove the result with $\lambda_k = \alpha_{k+1, (k+1)!}$, where $\alpha_{k+1, (k+1)!}$ is as promised by Theorem~\ref{mainthm2}. We note that this establishes the bound
	\[\lambda_k \ge (5(k+1)(k+1)!)^{-k} \ge k^{-2k^2}\]
	for $k \ge 3$; that $\lambda_k \ge k^{-2k^2}$ for all $k \in \N$ follows from the facts, respectively from~\cite{mader1} and~\cite{myhomeo}, that $\lambda_1 \ge 1$ and $\lambda_2 \ge 1/5$.

	Given a $k$-complex $\SS$, we first construct its canonical subdivision $\tilde \SS$ as described above. When this $k$-complex $\tilde \SS$ is viewed as a $(k+1)$-graph, it is clear that it is $(k+1)$-partite with partition classes $V_1, V_2, \dots, V_{k+1}$, where $V_1 = V(\SS)$ and for $2 \le t \le k+1$, $V_t$ consists of those new vertices $v_T$ introduced in $\tilde \SS$ for each $t$-set $T$ contained in some facet of $\SS$. Furthermore, $\tilde \SS$ is $((k+1)!)$-trace-bounded; indeed, it is easy to see, for each $2 \le t \le k+1$, that for every $v \in V_t$, we have $\Deg(v, \Tr_t(\tilde \SS)) = t! \le (k+1)!$.

	It follows from Theorem~\ref{mainthm2} that provided $n \ge n_0(\SS)$ is large enough, any $k$-complex on $n$ vertices with $n^{k+1-\lambda_k}$ facets contains $\tilde \SS$ as a subcomplex, and therefore, a homeomorph of $\SS$.
\end{proof}
\begin{figure}
	\centering
	\begin{tikzpicture}[scale=3.6, every node/.style={scale=1}]
		\draw[fill,gray,opacity=0.1] (0,0)--(1,0)--(1/2,0.866)--(0,0);
		\draw (0,0)--(1,0)--(1/2,0.866)--(0,0);
		\draw (0.5,-0.15) node[] {$\SS$};
		\draw (2.5,-0.15) node[] {$\tilde \SS$};

		\draw[fill,gray,opacity=0.1] (2,0)--(3,0)--(2.5,0.866)--(2,0);
		\draw (2,0)--(3,0)--(2.5,0.866)--(2,0);
		\draw (0,0) node[] {$\bullet$};
		\draw (1,0) node[] {$\bullet$};
		\draw (1/2,0.866) node[] {$\bullet$};

		\draw (2,0)--(2.5,0.289);
		\draw (3,0)--(2.5,0.289);
		\draw (2.5,0.866)--(2.5,0.289);

		\draw (2.25,0.866/2) node[blue] {$\bullet$};
		\draw (2.75,0.866/2) node[blue] {$\bullet$};
		\draw (2.5,0) node[blue] {$\bullet$};

		\draw (2.25,0.866/2)--(2.5,0.289)--(2.75,0.866/2);

		\draw (2.5,0.289)--(2.5,0);
		\draw (2,0) node[] {$\bullet$};
		\draw (3,0) node[] {$\bullet$};
		\draw (2.5,0.866) node[] {$\bullet$};
		\draw (2.5,0.289) node[red] {$\bullet$};

		\draw[->,  line width=0.3mm] (1.2,0.433) -- (1.8,0.433);
	\end{tikzpicture}
	\caption{The construction of $\tilde \SS$ for a $2$-complex $\SS$.}\label{fig:subd}
\end{figure}
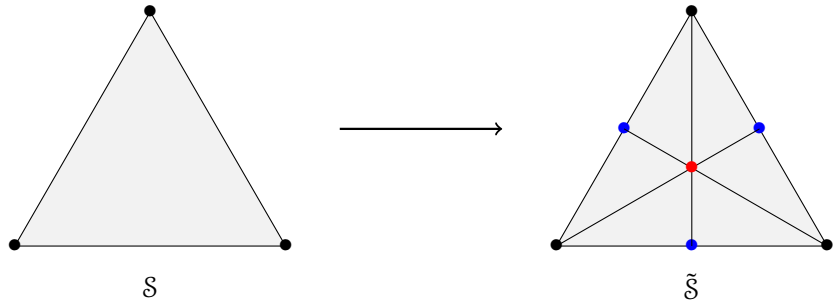

\section{Trace-bounded hypergraphs}\label{sec:proof}
We start with a lemma that says that if an $r$-partite $r$-graph has many edges and a small number of (small) subgraphs that are `marked' as being bad, then we may pass to an $(r-1)$-partite $(r-1)$-graph in one of its traces that has similar properties. To state this key lemma, we need a little set up, to which we now turn.

For each $r, d \in \N$, let $\HH(r, d)$ denote the (finite) family of all nonempty $r$-partite $r$-graphs with at most $d$ edges, taken up to isomorphism; recall our convention that the vertex set of an $r$-graph specified by its edge set alone is the span of its edges, whence $r \le v(J) \le rd$ for each $J \in \HH(r,d)$.

Suppose that $r \ge 2$, and let $G$ be an $r$-partite $r$-graph with partition classes $X_1, X_2, \dots, X_r$ each of size $n$. Suppose that for each $J \in \HH(r, d)$, a subset $\BB(J, G)$ of the copies of $J$ in $G$ have been \emph{marked (as being bad)}. We say that an $(r-1)$-graph $L \subset \Tr_{r-1}(G)$ with at most $d$ edges is \emph{$\beta$-bad with respect to $G$} if
\begin{enumerate}[label = {\bfseries{B\arabic{enumi}}}]
	\item\label{B1} either $|\Gamma(L, G)|\le n^{1-\beta}$, or
	\item\label{B2} if there exists some $J \in \HH(r, d)$ such that the number of marked copies of $J$ in $G$ containing $L$ is at least
	\[n^{-2\beta}\left(n^{1-\beta}\right)^{v(J)-v(L)-1}|\Gamma(L, G)|.\]
\end{enumerate}

The following lemma will be the workhorse that drives the proof of our main result.
\begin{lemma}\label{mainlem2}
	For a fixed $r \ge 2$, $d\in\N$, and $\eps, \delta >0$, the following holds for all sufficiently large $n \in \N$. Let $G$ be an $r$-partite $r$-graph  with partition classes $X_1, X_2, \dots, X_r$ each of size $n$ and $e(G) \ge 2^r n^{r-\eps}$ in which, for each $J \in \HH(r, d)$, there is a set $\BB(J, G)$ of at most $n^{v(J)-\delta}$ marked copies of $J$ in $G$. Then, for $\beta =\delta/(rd+1)$, there is an $(r-1)$-partite $(r-1)$-graph $G' \subset \Tr_{r-1}(G)$ with partition classes $X_1, X_2, \dots, X_{r-1}$ such that
	\begin{enumerate}
		\item $e(G') \ge 2^{r-1} n^{r-1-\eps}$, and
		\item for each $L \in \HH(r-1, d)$, the set $\BB(L, G')$ of copies of $L$ in $G'$ that are $\beta$-bad with respect to $G$ satisfies
		      \[|\BB(L, G')| \le n^{v(L)-\beta + 2\eps}.\]
	\end{enumerate}
\end{lemma}
\begin{proof}

	Choose a vertex $x\in X_r$ uniformly at random and let $G' = \LL(x, G) \subset \Tr_{r-1}(G)$ be the link of $x$ in $G$. It is clear that
	\beq{eg'}
	\E[e(G')] = e(G) / n \ge 2^{r} n^{r-1-\eps}.
	\enq

	For each $L \in \HH(r-1, d)$, let $P(L)$ be the set of copies $L'$ of $L$ in $G'$ with $|\Gamma(L', G)|\le n^{1-\beta}$. Next, for $L \in \HH(r-1,d)$ and $J\in \HH(r,d)$, we say that a copy $J'$ of $J$ in $G$ \emph{extends} a copy $L'$ of $L$ in $\Tr_{r-1}(G)$ if $\Tr_{r-1}(J')=L'$. Let $Q(L, J)$ be the set of copies $L'$ of $L$ in $G'$ which extend to at least
	\[ n^{-2\beta}\left(n^{1-\beta}\right)^{v(J)-v(L)-1}|\Gamma(L', G)|\]
	marked copies of $J$ belonging to $\BB(J, G)$. With these definitions in place,	we then have
	\[\BB(L, G')=P(L) \cup \left(\bigcup_{J\in\HH(r,d)}Q(L,J)\right) \]
	for each $L \in \HH(r-1, d)$; indeed, the first term above accounts for~\ref{B1} and the second for~\ref{B2}.

	First, we note that
	\beq{pl}
	\E[|P(L)|] \le n^{v(L)-\beta},
	\enq
	since each copy $L'$ of $L$ in $\Tr_{r-1}(G)$ with $|\Gamma(L', G)|\le n^{1-\beta}$ survives in $G'$ with probability at most $n^{-\beta}$, and the number of copies of $L$ in $\Tr_{r-1}(G)$ is trivially at most $n^{v(L)}$.

	Next, since $|\BB(J, G)|\le n^{v(J)-\delta}$ for each $J\in \HH(r,d)$, we have
	\[\sum_{L'\in Q(L, J) }  n^{-2\beta}\left(n^{1-\beta}\right)^{v(J)-v(L)-1}|\Gamma(L', G)|\le |\BB(J, G)| \le n^{v(J)-\delta},\]
	and by rearranging this, we get
	\begin{align*}
		\sum_{L'\in Q(L, J) }  \frac{|\Gamma(L', G)|}{n} & \le n^{2\beta}\left(n^{\beta-1}\right)^{v(J)-v(L)-1}n^{v(J)-1-\delta} \\
		                                                 & = n^{v(L)-\delta+\beta(1+v(J)-v(L))} \le n^{v(L)-\delta+\beta rd},
	\end{align*}
	where the last inequality uses the trivial facts that $v(J) \le rd$ and $v(L) \ge 1$. It follows that
	\beq{ql}
	\E[|Q(L,J)|] \le n^{v(L)-\delta+\beta rd}
	\enq
	for each $L \in \HH(r-1,d)$ and $J \in \HH(r,d)$.

	Putting the estimates~\eqref{pl} and~\eqref{ql} together, we get
	\beq{bbl}
	\E[|\BB(L, G')|] \le n^{v(L)-\beta}+|\HH(r,d)| n^{v(L)-\delta+\beta rd} = C_1 n^{v(L)-\beta},
	\enq
	where $C_1 = (1 + |\HH(r,d)|)$, the last equality following from the fact that $\delta= \beta(rd+1)$.

	To finish, we set $C_2 = |\HH(r-1,d)|$, and combine~\eqref{eg'} and~\eqref{bbl} to get
	\[
		\E\left[\frac{e(G')}{2^{r-1}n^{r-1-\eps}}- 1 - \frac{1}{C_2}\sum_{L\in \HH(r-1,d)} \frac{|\BB(L, G')|}{C_1  n^{v(L)-\beta}}  \right] \ge 0;
	\]
	consequently, there is at least one vertex in $X_r$ whose link $G'$ has the following properties:
	\begin{enumerate}
		\item $e(G') \ge 2^{r-1} n^{r-1-\eps} $, and
		\item for every $L \in \HH(r-1,d)$, we have
		      \[|\BB(L, G')|  \le \frac{C_1 C_2 n^{v(L)-\beta}e(G')}{2^{r-1} n^{r-1-\eps}} \le 2^{1-r}C_1 C_2 n^{v(L)-\beta +\eps} \le n^{v(L)-\beta + 2\eps};  \]
		      here, we use the facts that $e(G') \le n^{r-1}$, that $\eps > 0$, and that $n$ is sufficiently large.
	\end{enumerate}
	Such an $(r-1)$-graph $G'$ has all the properties we require, proving the lemma.
\end{proof}

With Lemma~\ref{mainlem2} in hand, we are now ready to prove our second main result.
\begin{proof}[Proof of Theorem~\ref{mainthm2}]
	Let $H$ be a $d$-trace-bounded $r$-partite $r$-graph  with partition classes $Y_1, Y_2, \dots, Y_r$. For convenience, we prove that any large $r$-graph with sufficiently many edges on a vertex set whose cardinality is \emph{divisible by $r$} must contain a copy of $H$; of course, this divisibility assumption makes no material difference beyond allowing us to drop floor and ceiling signs. We shall prove the result with the precise value of
	\[\alpha_{r,d} = \frac{1}{10d}\left(\frac{1}{rd+1}\right)^{r-2},\]
	noting that $\alpha_{r,d} \ge (5rd)^{1-r}$ for all $r \ge 2$ and $d\ge1$.

	Given an $r$-graph on $rn$ vertices with at least $(rn)^{r - \alpha_{r,d}}$ edges, then provided $n$ is sufficiently large, we may, by Proposition~\ref{triedges}, pass to an $r$-partite subgraph with partition classes $X_1, X_2, \dots, X_r$ each of size $n$ containing $2^rn^{r-\eps}$ edges, for some
	\[0 < \eps \le \frac{1}{9d}\left(\frac{1}{rd+1}\right)^{r-2};\]
	we shall only work with this $r$-partite $r$-graph, which we call $G$, in what follows. Our goal now is to show that we are guaranteed to find a copy of $H$ in $G$ provided $n \ge n_0(H)$ is sufficiently large.

	Our proof proceeds in two stages. In the first stage, we shall inductively construct a sequence of $i$-partite $i$-graphs $G_i \subset \Tr_i(G)$ with partition classes $X_1, X_2, \dots, X_{i}$ for $r-1 \ge i \ge 1$, with $G_{i}$ being constructed by feeding $G_{i+1}$ into Lemma~\ref{mainlem2}. To accomplish this iterative construction, we need to find a suitable $G_{r-1}$ from which to start, which we do as follows.

	\begin{claim}\label{mainlem1}
		There is an $(r-1)$-partite $(r-1)$-graph $G_{r-1} \subset \Tr_{r-1}(G)$ with partition classes $X_1, X_2, \dots, X_{r-1}$ such that
		\begin{enumerate}
			\item $e(G_{r-1}) \ge 2^{r-1} n^{r-1-\eps} $, and
			\item for each $J \in \HH(r-1,d)$, the set $\BB(J, G_{r-1})$ of copies of $J$ in $G_{r-1}$ that are contained in the link of fewer than $v(H)$ different vertices of $X_r$ in $G$ satisfies
			      \[|\BB(J, G_{r-1})| \le n^{v(J)-1/2}.\]
		\end{enumerate}
	\end{claim}

	\begin{proof}
		The proof mirrors that of Lemma~\ref{mainlem2}, but involves less work since we are aiming to accomplish less. Indeed, choose a vertex $x\in X_r$ uniformly at random and let $G' = \LL(x, G) \subset \Tr_{r-1}(G)$ be the link of $x$ in $G$. As before, we clearly have
		\beq{edges}
		\E[e(G')] = e(G)/n =  2^{r} n^{r-1-\eps}.
		\enq
		For each $J \in \HH(r-1,d)$, the probability that a copy of $J$ in $\Tr_{r-1}(G)$ contained in the link of fewer than $v(H)$ different vertices of $X_r$ survives in $G'$ is at most $v(H)/ n$, so it follows that $\BB(J, G')$ satisfies
		\beq{badjs}
		\E [ |\BB(J, G')|] \le v(H) n^{v(J)-1}.
		\enq
		Setting $C = |\HH(r-1,d)| v(H) $ and putting~\eqref{edges} and~\eqref{badjs} together, we get
		\[\E \left[\frac{e(G')}{2^{r-1}n^{r-1-\eps}}- 1 - \frac{1}{C} \sum_{J \in \HH(r-1,d)}\frac{|\BB(J, G')|}{n^{v(J)-1}}\right] \ge 0.\]

		Consequently, there is at least one vertex in $X_r$ whose link $G'$ has the following properties:
		\begin{enumerate}
			\item $e(G') \ge 2^{r-1}n^{r-1-\eps} $, and
			\item for every $J \in \HH(r-1,d)$, we have
			      \[|\BB(J, G')|  \le \frac{C n^{v(J)-1}e(G')}{2^{r-1}n^{r-1-\eps}} \le 2^{1-r} C n^{v(J)-1 +\eps} \le n^{v(J)-1/2};  \]
			      here, we use the facts that $e(G') \le n^{r-1}$, that $\eps<1/2$, and that $n$ is sufficiently large.
		\end{enumerate}
		The claim follows by taking $G_{r-1}$ to be a link $G'$ with the above properties.
	\end{proof}

	Let $G_{r-1}$ be the $(r-1)$-graph promised by Claim~\ref{mainlem1} and set $\delta_{r-1}=1/2$. We know that
	\begin{enumerate}
		\item $e(G_{r-1}) \ge 2^{r-1} n^{r-1-\eps}$, and
		\item for each $J \in \HH(r-1,d)$, the set $\BB(J, G_{r-1})$ of copies of $J$ in $G_{r-1}$ contained in fewer than $v(H)$ distinct links in $G$ satisfies $|\BB(J, G_{r-1})| \le n^{v(J)-\delta_{r-1}}$.
	\end{enumerate}

	For $r-2 \ge i \ge 1$, having constructed $G_{i+1}$ with
	\begin{enumerate}
		\item $e(G_{i+1}) \ge 2^{i+1}n^{i+1-\eps}$ along with,
		\item for each $J \in \HH(i+1, d)$, a set $\BB(J, G_{i+1})$ of at most $n^{v(J)-\delta_{i+1}}$ bad copies of $J$ in $G_{i+1}$,
	\end{enumerate}
	we apply Lemma~\ref{mainlem2} to $G_{i+1}$ to construct $G_i$ such that
	\begin{enumerate}
		\item $e(G_i) \ge 2^{i}n^{i-\eps}$, and
		\item for each $J \in \HH(i, d)$, the set $\BB(J, G_i)$ of copies of $J$ in $G_i$ that are $\beta_i$-bad with respect to $G_{i+1}$ satisfies
		      \[|\BB(J, G_i)| \le n^{v(J)-\delta_i},\]
		      where $\beta_i=\delta_{i+1}/((i+1)d+1)$ and $\delta_{i}=\beta_i-2\eps$.
	\end{enumerate}

	Since $\delta_{r-1}=1/2$ and $\delta_i={\delta_{i+1}}/{((i+1)d+1)}-2\eps\ge {\delta_{i+1}}/{(rd+1)}-2\eps$ for each $r-2 \ge i \ge 1$, we get that
	\[
		\delta_1\ge \frac12\left(\frac{1}{rd+1}\right)^{r-2}-2\eps \left(\frac{1- (1/(rd+1))^{r-2}}{1- 1/(rd+1)} \right)\ge \frac12 \left(\frac{1}{rd+1}\right)^{r-2} -3\eps > d\eps,
	\]
	where the last inequality above uses the fact that $\eps \le (1/9d) (1/(rd+1))^{r-2}$. Consequently, we have
	\begin{enumerate}
		\item $1/2 = \delta_{r-1} \ge \delta_{r-2} \ge \dots \ge \delta_1 > d \eps$, and
		\item $1/2 \ge \beta_{r-2} \ge \beta_{r-3} \ge \dots \ge \beta_1 > 0$.
	\end{enumerate}

	We are now ready for the second stage of the proof where we embed $H$ into $G$. First, we may assume that $H$ has no isolated vertices, i.e., any vertices $y$ for which $\Deg(y, H) = 0$; indeed, if $H$ has isolated vertices, we may embed the rest of $H$ into $G$ first, and then embed the isolated vertices into $G$ arbitrarily provided $n$ is sufficiently large. By assuming $H$ has no isolated vertices, we have $\Tr_r(H) = H$ and furthermore, the link of any vertex in $\Tr_j(H)$ is nonempty for each $2 \le j \le r$.

	For $1 \le j \le r-1$, we shall sequentially construct injective maps $\phi_j : Y_j \to X_j$ in such a way that
	\begin{enumerate}
		\item each induced map $\Phi_{j}: Y_1 \cup Y_2 \cup \dots \cup Y_j \to X_1 \cup X_2 \cup \dots \cup X_j$ is an embedding of $\Tr_{j}(H)$ into $G_{j}$, and
		\item for each nonempty subgraph $J$ of $\Tr_j(H)$ with at most $d$ edges, no copy of $J$ in the image of $\Phi_{j}$ is in $\BB(J, G_{j})$.
	\end{enumerate}

	Observe that $G_1$ is a 1-graph on $X_1$, which is just a subset of $X_1$, so $v(G_1) = e(G_1)$, and we have
	\[m = v(G_1) = e(G_{1})\ge 2n^{1-\eps}.\]
	We construct $\phi_1: Y_1 \to X_1$ in such a way that for every $J \in \HH(1,d)$, no copy of $J$ in $\phi_1 (\Tr_1(H))$ is in $\BB(J, G_1)$, i.e., is $\beta_1$-bad with respect to $G_2$. By construction, we have $|\BB(J, G_1)| \le n^{v(J)-\delta_1}$ for each $J \in \HH(1,d)$. Note  that $\HH(1,d)$ consists of $d$ elements, namely one set of cardinality $t$ for each $1 \le t \le d$, and also note that the number of $\beta_1$-bad $t$-sets in $G_1$ is $o(m^t)$ for each $1 \le t\le d$ since $\delta_1/t \ge \delta_1/d > \eps$. It follows that the number of problematic $d$-sets in $G_1$, namely those containing a $\beta_1$-bad $t$-set for some $1 \le t \le d$, is $o(m^d)$, so we may choose a subset of $X_1$ of size $|Y_1|$ which does not contain any such problematic $d$-set provided $n$ is sufficiently large, as can be seen, for example, by applying a bound of de Caen~\cite{deCaen} to the $d$-graph of all problematic $d$-sets. In other words, we can choose a subset $S$ of $X_1$ of size $|Y_1|$ in such a way that for each $J \in \HH(1,d)$, no copy of $J$ in $\BB(J, G_1)$ is contained in $S$; we take $\phi_1$ to be any injective map from $Y_1$ to $S$.

	For $2 \le j \le r-2$, suppose that we have constructed injective maps $\phi_1, \phi_2, \dots, \phi_{j-1}$ as above. We now extend $\Phi_{j-1}$ to $\Phi_j$ by defining a suitable map $\phi_j:Y_j\to X_j$ randomly as follows.

	Since $H$ is $d$-trace-bounded, each vertex $y\in Y_j$ has degree at most $d$ in $\Tr_j(H)$. Given a vertex $y\in Y_j$, let $L(y) = \LL(y, \Tr_j(H))$ be the link of $y$ in $\Tr_j(H)$, so that $L(y)$ is a nonempty subgraph of $\Tr_{j-1}(H)$ with at most $d$ edges. Inductively, we know that $\Phi_{j-1}(L(y))$ is a subgraph of $G_{j-1}$. We choose $\phi_j(y)$ uniformly at random from the common neighbourhood $\Gamma(\Phi_{j-1}(L(y)), G_j) \subset X_j$. By choosing, for each $y \in Y_j$, the image $\phi_j(y)$ of $y$ from the set $\Gamma(\Phi_{j-1}(L(y)), G_j)$, we have ensured that the induced map $\Phi_j$ is a homomorphism from $\Tr_j(H)$ into $G_j$. We claim that with positive probability, both of the following events hold:
	\begin{enumerate}[label = {\bfseries{E\arabic{enumi}}}]
		\item\label{AA} the induced map $\Phi_{j}$ is injective, i.e., is an embedding of $\Tr_{j}(H)$ into $G_{j}$, and
		\item\label{BB} for each nonempty subgraph $J$ of $\Tr_j(H)$ with at most $d$ edges, no copy of $J$ in the image of $\Phi_{j}$ is in $\BB(J, G_{j})$.
	\end{enumerate}

	To deal with~\ref{AA}, note that for any vertex $y\in Y_j$, the $(j-1)$-graph $L(y) = \LL(y, \Tr_j(H))$ is nonempty and has at most $d$ edges, and inductively, its image $\Phi_{j-1}(L(y))$ is not in $\BB(L(y), G_{j-1})$, namely the set of copies of $L(y)$ in $G_{j-1}$ that are $\beta_{j-1}$-bad with respect to $G_j$, so it follows that
	\beq{choices}|\Gamma(L(y), G_j)|\ge n^{1-\beta_{j-1}}.
	\enq
	Since $\beta_{j-1} \le 1/2$, the probability that $\phi_j$, and therefore $\Phi_j$, fails to be injective is easily seen to be $o(1)$, whence we certainly have $\Prob(\text{\textbf{E1}}) > 1/ 2$ provided $n$ is sufficiently large.

	To address~\ref{BB}, we argue as follows. Let $J$ be a nonempty subgraph of $\Tr_j(H)$ with at most $d$ edges, and let $L=\Tr_{j-1}(J)$. The probability of the event that $\Phi_j(J) \in \BB(J ,G_j)$ may be bounded above as follows. By the bound in~\eqref{choices}, the number of choices for $\phi_{j}$ on $V(J) \setminus V(L)$ is at least $(n^{1-\beta_{j-1}})^{v(J)-v(L)}$. On the other hand, since $L$ is nonempty with at most $d$ edges, we know inductively that $\Phi_{j-1}(L) \notin \BB(L, G_{j-1})$. Therefore, $\Phi_{j-1}(L)$ is not $\beta_{j-1}$-bad with respect to $G_j$, from which it follows that $\Phi_{j-1}(L)$ is contained in at most
	\[n^{-2\beta_{j-1}}\left(n^{1-\beta_{j-1}}\right)^{v(J)-v(L)-1}|\Gamma(\Phi_{j-1}(L), G_j)|\] copies of $J$ in $G_j$ that belong to $\BB(J, G_j)$. Hence, the probability of $\Phi_j(J) \in \BB(J ,G_j)$ is at most
	\[\frac{n^{-2\beta_{j-1}}\left(n^{1-\beta_{j-1}}\right)^{v(J)-v(L)-1}|\Gamma(\Phi_{j-1}(L), G_j)|}{\left(n^{1-\beta_{j-1}}\right)^{v(J)-v(L)}}\le \frac{n^{1-2\beta_{j-1}}}{n^{1-\beta_{j-1}}}=n^{-\beta_{j-1}}=o(1),\]
	where we use the facts that $|\Gamma(\Phi_{j-1}(L), G_j)| \le n$ and that $\beta_{j-1} > 0$. Summing this estimate over the $O(1)$ choices of nonempty subgraphs $J$ of $\Tr_j(H)$ with $e(J) \le d$ shows that $\Prob(\text{\textbf{E2}}) > 1/ 2$ provided $n$ is sufficiently large as well, which together with the fact that $\Prob(\text{\textbf{E1}}) > 1/ 2$ establishes the existence of an appropriate $\phi_j$.

	Now, we finish by extending $\Phi_{r-1}$ to an embedding $\Phi_r: Y_1 \cup Y_2 \cup \dots Y_r \to X_1 \cup X_2 \cup \dots X_r$ of $H$ into $G$ by defining a final injective map $\phi_r:Y_r \to X_r$. Recall that for each $J \in \Tr_{r-1}(H)$, the set $\BB(J, G_{r-1})$ consists of those copies $J'$ of $J$ in $G_{r-1}$ for which the number of vertices $x \in X_r$ whose link $\LL(x, G)$ contains $J'$ as a subgraph is at most $v(H)$. We may now define $\phi_r$ by greedily picking, for each $y \in Y_r$, a distinct vertex $\phi_r(y) \in \Gamma(\Phi_{r-1}(\LL(y, H)), G) \subset X_r$; this is always possible since $\Phi_{r-1}(\LL(y, H)) \notin \BB(\LL(y, H), G_{r-1})$. It follows that $\Phi_r(H)$ is a copy of $H$ in $G$, completing the proof.
\end{proof}

\section{Conclusion}\label{sec:conc}
A number of open problems remain, and we conclude by highlighting those that we think are particularly deserving of attention.

With regard to homeomorphs, the outstanding problem is to determine the optimal values of the exponents $\lambda_k$ in Theorem~\ref{mainthm1} for each $k \ge 2$. As mentioned earlier, even the optimal value of $\lambda_2$ is not known, though there are good reasons to expect it to be $1/2$. In higher dimensions, we are only able to speculate: could it be that for each $k \ge 2$, the optimal value of $\lambda_k$ is precisely $\lambda(S^k)$, where $S^k$ is (any triangulation of) the $k$-sphere? This is indeed the underlying mechanism behind the prediction of the value of $1/2$ in dimension two, but we do not know, nor do we have a guess for, the value of $\lambda(S^k)$ for any $k \ge 3$; for example, a combination of a random construction and a generalisation of the argument of the Brown, Erd\H{o}s and S\'os~\cite{bes} shows that $1/4 \le \lambda(S^3) \le 1/3$, but we have no reason to think either bound reflects the truth.

With regard to trace-bounded hypergraphs, the main problem again is to determine the optimal values of the exponents $\alpha_{r,d}$ in Theorem~\ref{mainthm2} for each $r \ge 3$ and $d\in \N$. As remarked upon earlier, we know that the optimal value of $\alpha_{2,d}$ is $1/d$, but even formulating a natural guess for the optimal value of $\alpha_{r,d}$ when $r \ge 3$, let alone proving it, would be of some interest.


\bibliographystyle{amsplain}


\begin{dajauthors}
\begin{authorinfo}[jl]
  Jason Long\\
  Cambridge CB1\thinspace1JH, UK\\
  jasonlong272\imageat{}gmail\imagedot{}com
\end{authorinfo}
\begin{authorinfo}[bn]
  Bhargav Narayanan\\
  Department of Mathematics, Rutgers University\\
  Piscataway, NJ, USA\\
  narayanan\imageat{}math\imagedot{}rutgers\imagedot{}edu \\
  \url{https://sites.math.rutgers.edu/~narayanan/}
\end{authorinfo}
\begin{authorinfo}[cy]
  Corrine Yap\\
  Department of Mathematics, Rutgers University\\
  Piscataway, NJ, USA\\
  corrine\imagedot{}yap\imageat{}rutgers\imagedot{}edu\\
  \url{http://www.corrineyap.com/}
\end{authorinfo}

\end{dajauthors}

\end{document}